\documentclass[a4paper]{amsart}
\usepackage{amssymb, enumitem}
\usepackage[all,cmtip,arrow,matrix,curve]{xy}
\usepackage{aliascnt, graphicx,hyperref}
\usepackage{mathtools}

\usepackage[capitalise, nameinlink, noabbrev, nosort]{cleveref} 

\theoremstyle{plain}
\newtheorem{lma}{Lemma}[section]
\crefname{lma}{Lemma}{Lemmata}
\newtheorem{thm}[lma]{Theorem}
\crefname{thm}{Theorem}{Theorems}
\newtheorem{cor}[lma]{Corollary}
\crefname{cor}{Corollary}{Corollaries}
\newtheorem{prp}[lma]{Proposition}
\crefname{prp}{Proposition}{Propositions}

\theoremstyle{definition}
\newtheorem{pgr}[lma]{}
\crefname{pgr}{Paragraph}{Paragraphs}
\newtheorem{dfn}[lma]{Definition}
\crefname{dfn}{Definition}{Definitions}

\theoremstyle{remark}
\newtheorem{rmk}[lma]{Remark}
\crefname{rmk}{Remark}{Remarks}
\newtheorem{exa}[lma]{Example}
\crefname{exa}{Example}{Examples}
\newtheorem{qst}[lma]{Question}
\crefname{qst}{Question}{Questions}

\crefname{ntn}{Notation}{Notations}

\newtheorem*{thm*}{Theorem}
\newtheorem*{qst*}{Question}
\newtheorem*{cor*}{Corollary}

\newcommand{\andSep}{\,\,\,\text{ and }\,\,\,}
\newcommand{\axiomO}[1]{(O#1)}
\newcommand{\CatCu}{\ensuremath{\mathrm{Cu}}}
\newcommand{\CuSgp}{$\CatCu$-sem\-i\-group}
\newcommand{\CuMor}{$\CatCu$-mor\-phism}
\newcommand{\soft}{{\rm{soft}}}

\def\today{\number\day\space\ifcase\month\or   January\or February\or
   March\or April\or May\or June\or   July\or August\or September\or
   October\or November\or December\fi\   \number\year}

\newcommand{\NN}{{\mathbb{N}}}
\newcommand{\QQ}{{\mathbb{Q}}}
\newcommand{\id}{{\mathrm{id}}}

\DeclareMathOperator{\Cu}{Cu}

\newcommand{\ca}{$\mathrm{C}^*$-algebra}

\newcommand{\stHom}{${}^*$-homomorphism}

\title{Pure ${}^*$-homomorphisms}
\date{\today}

\author{Joan Bosa, Eduard Vilalta}

\address{Joan Bosa,
Departamento de Matem\'{a}ticas,
Universidad de Zaragoza,
50009 Zaragoza, Zaragoza, Spain.}
\email{jbosa@unizar.es}
\urladdr{personal.unizar.es/jbosa/}

\address{Eduard Vilalta, 
	Department of Mathematical Sciences,
	Chalmers University of Technology and University of Gothenburg, 
	Chalmers Tvärgata 3, SE-412 96 Gothenburg, Sweden}
\email[]{vilalta@chalmers.se}
\urladdr{www.eduardvilalta.com}

\thanks{All authors were partially supported by MINECO (grant No.\ PID2020-113047GB-I00/AEI/10.13039/501100011033), and by the Comissionat per Universitats i Recerca de la Generalitat de Catalunya (grant No. 2021SGR01015). The first author was also partially supported by the Consolidación Investigadora grant (CNS2022-135340) provided by Ministerio de Ciencia, Innovación y Universidades (Gobierno de España). The second author was also supported by the Fields Institute for Research in Mathematical Sciences and by the Knut and Alice Wallenberg Foundation (KAW 2021.0140)}

\makeatletter
\@namedef{subjclassname@2020}{\textup{2020} Mathematics Subject Classification}
\makeatother

\subjclass[2020]%
{Primary
46L05; 
Secondary
19K14, 
46L80. 
}
\date{\today}

\begin{document}

\begin{abstract}
We introduce and study a notion of pureness for \stHom{s} and, more generally, for cpc. order-zero maps. After providing several examples of pureness, such as \textquoteleft{}$\mathcal{Z}$-stable\textquoteright{}-like maps, we focus on the question of when pure maps factor through a pure \ca{}.

We show that, up to Cuntz equivalence, any composition of two pure maps factors through a pure object. This is used to obtain several factorization results at the level of \ca{s}.

\end{abstract}
\maketitle

\section{Introduction}
 The classification of unital separable simple nuclear non-elementary \ca{s} by their K-theory and tracial data, also known as the Elliott classification program, has been one of the major developments in the theory of operator algebras in the past 40 years. A landmark achievement, obtained as the collaborative effort of many hands and decades of work (see, among many others, \cite{EllGonLinNiuTA,TikWhiWin17QDNuclear,Win12NuclDimZstable}), states that any such \ca{} $A$ can be classified by the so-called Elliott invariant as long as $A$ is $\mathcal Z$-stable (i.e. it absorbs the Jiang-Su algebra tensorially) and satisfies the universal coefficient theorem (UCT). Some of the first results in the Elliott classification program relied heavily on the inductive limit presentation of the algebras under consideration; however, the modern approach to classification exploits conditions that are more abstract in nature. One such condition is Winter's notion of \emph{pure C*-algebras}, which were defined in \cite{Win12NuclDimZstable} as those algebras that are both almost divisible and almost unperforated. This notion is deeply connected to $\mathcal Z$-stability, as stated explicitly in the famous Toms-Winter conjecture: For unital separable simple nuclear non-elementary \ca{s}, the conditions of almost unperforation, $\mathcal Z$-stability and finite nuclear dimension should all coincide. The conjecture is by now largely a theorem \cite{CasEviTikWhiWin21NucDimSimple,BBSTWW,Ror04StableRealRankZ,Win12NuclDimZstable}, with the only remaining implication being if almost unperforation implies $\mathcal Z$-stability. If true, pureness and $\mathcal Z$-stability agree for all unital separable simple nuclear \ca{s}. However, it should be noted that $\mathcal Z$-stability and pureness do not agree in general, with the latter then becoming an important regularity property on its own right; see \cite{AntPerThiVil24}.
 
 The current approaches to the classification program classify \ca{s} by
classifying maps. In such results, strong conditions are only imposed either on the domain or codomain, while the assumptions on the other side tend to be milder. One important example of this phenomenon is Robert's classification of \stHom{s} from $1$-dimensional NCCW-complexes with trivial $K_1$-group (and their inductive limits) to stable rank one \ca{s} via the \emph{Cuntz semigroup} \cite{Rob12Class}. This semigroup is a rich invariant for \ca{s}, which plays a crucial role in  both Elliott's program and the Toms-Winter conjecture.  Robert's theorem has been an important tool in recent classification results, and can be regarded as an example where the result for \stHom{s} is much more powerful than its induced result for \ca{s}. A recent groundbreaking development along these lines can be found in \cite{CGSTW23ClassHom}, where unital embeddings from unital separable nuclear \ca{s} satisfying the UCT to unital simple separable nuclear $\mathcal{Z}$-stable \ca{s} are classified. Further, a current trend in such results is to move the conditions on the domain or codomain to the maps themselves. This can be seen, for example, in the definition of $\mathcal{O}_2$-stable \stHom{s} \cite{Gabe2020ANewProofKirchberg}, in the study of $\mathcal{O}_\infty$-stable \stHom{s} run in \cite{BGSW22_infty}, in the introduction of real rank zero inclusions \cite{GabNea23RRZero}, and many others.

Inspired by this modern approach to classification ---as well as by its importance as a regularity property--- in this paper we generalize the notion of pureness to maps between \ca{s}. Extending the original definition, we say that a cpc. order-zero map (in particular, a \stHom{}) $\theta\colon A\to B$ is \emph{pure} if it is both almost unperforated and almost divisible, in their suitable versions (\cref{dfn:ZMultPure}).



Our first examples of pureness arise from the study of \textquoteleft{}$\mathcal{Z}$-stable\textquoteright{}-like maps (a notion that we do not define).  Recall from \cite[Proposition 4.4]{Kir06CentralSeqPI} that a unital separable \ca{} $A$ is $\mathcal{Z}$-stable if and only if $\mathcal{Z}$ embeds unitally into $A_\omega \cap A'$ (for a free ultrafilter $\omega$), which is equivalent to $\mathcal{Z}$ embedding unitally to $A_\omega \cap A'\cap S'$ for any separable sub-\ca{} $S\subseteq A_\omega \cap A'$. In particular, the unit in $A_\omega \cap A'\cap S'$ is almost divisible for each $S$. In our setting, we get the following:

\begin{thm}[cf. \ref{prp:ZstabImpPure}]\label{thmIntro:1}
 Let $A$ be $\sigma$-unital, and let $\theta\colon A\to B$ be a \stHom{}. Assume that $1\in B_\omega\cap\theta (A)'/{\rm Ann}(\theta (A))$ is almost divisible. Then, $\theta$ is pure.
 
 Further, if $1\in B_\omega\cap(\theta (A)\cup S)'/{\rm Ann}(\theta (A)\cup S)$ is almost divisible for every separable sub-\ca{} $S\subseteq B_\omega\cap\theta (A)'$, there exists a pure sub-\ca{} $C\subseteq B_\omega$ such that $\theta (A)\subseteq C \subseteq B_\omega$.
\end{thm}

The second part of \cref{thmIntro:1} says that, up to passing to the ultraproduct, certain pure \stHom{s} factor through a pure \ca{}. In the general setting, a central question in the study of regularity properties for maps is which \stHom{s} with a certain property factor, up to Murray-von Neumann equivalence (\cref{prp:ApproxMvNMult}), through a \ca{} with said property; see the comments before \cref{qst:Pure1} for a more in-depth discussion. Restricted to our case, the question is:

\begin{qst}[\ref{qst:Pure1}-\ref{qst:Pure2}]\label{qst:IntroQst}
  Let $\theta\colon A\to B$ be a \stHom{}. Is $\theta$ pure if and only if the map $\iota_w\circ\theta\colon A\to B_\omega$ factors, up to Murray-von Neumann equivalence, through a pure \ca{}?

  More generally, does there exist $n\in\NN$ such that, for any tuple $\theta_1,\ldots ,\theta_n$ of pairwise composable pure \stHom{s}, the composition $\iota_w\circ \theta_n\circ\cdots\circ \theta_1$ factors up to Murray-von Neumann equivalence through a pure \ca{}?
 \end{qst}

 We start our study defining the notion of pureness and asserting \cref{thmIntro:1} in \cref{sc:PuereStHom}, where we also provide a number of examples. In \cref{subsec:PermanenceProp} we establish several permanence properties of pureness that are used throughout the paper. In \cref{subsec:McDuffness}, the main section of the paper, we combine all the previous results to provide an answer to \cref{qst:IntroQst}. To do so, we exploit the structure of (abstract) Cuntz semigroup morphisms that are both almost divisible and almost unperforated. Our main technical result (\cref{prp:McDuffZ}) says that, at the level of Cuntz semigroups, any composition of two pure maps factors through a pure object. Restricted to \ca{s}, the result reads as follows:

\begin{thm}[cf. \ref{thm:MainCuA}]
 Let $\theta_1\colon A_1\to A_2$ and $\theta_2\colon A_2\to B$ be pure *-homo\-mor\-phisms. Then, there exists a \CuMor{} $\beta$ such that the following diagram commutes
 \[
 \xymatrix{
     \Cu (A_1) \ar[rr]^{\Cu (\theta_2\theta_1)} \ar[rd]_{-\otimes 1} && \Cu (B) \\
     & \Cu (A_1)\otimes \Cu (\mathcal{Z}) \ar[ru]_{\beta}
   } 
\]
\end{thm}

\cref{thm:MainCuA} is in fact more general. For the previous result to hold, one only needs $\theta_1,\theta_2$ to be cpc.~order-zero maps, with $\theta_1$ almost divisible and $\theta_2$ almost unperforated; see \cref{dfn:ZMultPure} for details. This generality provides us with many applications where the above statement can be used. Indeed, one obtains a complete answer to \cref{qst:IntroQst} when $A_1$ is AF and $B$ has stable rank one.

\begin{cor}[cf. \ref{prp:AFPureFactor}]
Let $A_1$ be a unital AF-algebra, and let $B$ be a unital \ca{} of stable rank one. Let $\theta_1\colon A_1\to A_2$ and $\theta_2\colon A_2\to B$ be unital, pure \stHom{s}. Then, $\theta_2\theta_1$ factors, up to approximately unitarily equivalence, through $A_1\otimes\mathcal{Z}$.
\end{cor}

Further, if $B$ has strict comparison, one can set $\theta_2=id_B$ to answer \cref{qst:IntroQst} for a single \stHom{} in the above situation. We state this in Corollary \ref{cor:AFStrict}.

The last section of the paper, \cref{subsec:WMult}, is devoted to two types of pure maps: $q$-rational \stHom{s} (\cref{dfn:RatCuMor}) and soft, pure \stHom{s} (see \cref{dfn:WMult}). Here, we deduce analogues of \cref{thm:MainCuA} for other types of tensorial absorption. Again, such results are also valid for cpc. order-zero maps, at the expense of $\beta$ being only a generalized \CuMor{}.

\begin{thm}[cf. \ref{thm:MainCuAMq}, \ref{thm:MainCuAW}]\label{thmInt:MqAW}
 Let $\theta_1\colon A_1\to A_2$ and $\theta_2\colon A_2\to B$ be $q$-rational (resp. soft, pure) \stHom{s}. Then, there exists a \CuMor{} $\beta$ such that the following diagram commutes
 \[
 \xymatrix{
     \Cu (A_1) \ar[rr]^{\Cu (\theta_2\theta_1)} \ar[rd]_{-\otimes 1} && \Cu (B) \\
     & \Cu (A_1\otimes \mathcal{D}) \ar[ru]_{\beta}
   } 
\]
where $\mathcal{D}$ is the UHF-algebra $M_q$ (resp. the Jacelon-Razak algebra $\mathcal{W}$).
\end{thm}

Combining the results above with Robert's classification result, we obtain:

\begin{cor}[cf. \ref{cor:RobMq}
]Let $\theta_1\colon A_1\to A_2$ and $\theta_2\colon A_2\to B$ be two unital $q$-rational \stHom{s}. Assume that $A_1$ is stably isomorphic to an  inductive limit of $1$-dimensional NCCW-complexes with trivial $K_1$-group, and that $B$ is of stable rank one. Then, $\theta_2\theta_1$ factors, up to approximate unitary equivalence, through $A_1\otimes M_q$.
\end{cor}

\subsection*{Acknowledgments} This work was carried out during two visits of EV to JB at Universidad de Zaragoza. He is grateful to the institution for its hospitality.

Both authors would also like to thank Hannes Thiel for comments on a first draft of the paper.


\section{Preliminaries}

\begin{pgr}[The Cuntz semigroup]
Given two positive elements $a,b$ in a \ca{} $A$, we write $a\precsim b$ whenever $a$ is \emph{Cuntz subequivalent} to $b$, that is, whenever there exists $(r_n)_n\subseteq A$ such that $a=\lim_n r_n b r_n^*$. Further, we say that $a$ is \emph{Cuntz equivalent} to $b$, and write $a\sim b$, whenever $a\precsim b$ and $b\precsim a$.

The \emph{Cuntz semigroup} of $A$ is defined to be the set $(A\otimes\mathcal{K})_+/\sim$, equipped with the addition induced by diagonal addition and the partial order induced by $\precsim$. We denote this monoid by $\Cu (A)$; see \cite{CowEllIva08CuInv}.

As shown in \cite{WinZac09CpOrd0}, any completely positive, contractive (cpc.), order-zero map $\varphi\colon A\to B$ (for example, any \stHom{}) induces a well-defined, partially ordered, monoid morphism $\Cu(\varphi )\colon \Cu (A)\to \Cu (B)$ given by $\Cu(\varphi )([a])=[\varphi (a)]$.
\end{pgr}

\begin{pgr}[Abstract Cuntz semigroups]
As defined in \cite{CowEllIva08CuInv}, a positively ordered monoid $S$ is a \emph{\CuSgp{}} if it satisfies the following four conditions:
\begin{itemize}
\item[\axiomO{1}] every increasing sequence has a supremum;
\item[\axiomO{2}] every element is the supremum of a $\ll$-increasing sequence;
\item[\axiomO{3}] $x'+y'\ll x+y$ whenever $x'\ll x$ and $y'\ll y$;
\item[\axiomO{4}] $\sup_n (x_n+y_n)=\sup_n x_n +\sup_n y_n$ for any pair of increasing sequences $(x_n),(y_n)$ in $S$,
\end{itemize}
where one writes $x\ll y$ if, for any increasing sequence $(z_n)_n$ with supremum greater than or equal to $y$, there exists $n\in\NN$ such that $x\leq z_n$.

A monoid morphism $S\to T$ between \CuSgp{s} is said to be a \emph{generalized \CuMor{}} if it preserves both the order and suprema of increasing sequences. A \emph{\CuMor{}} is any generalized \CuMor{} that also preserves the $\ll$-relation. Any cpc. order-zero map induces a generalized \CuMor{}, while any \stHom{} induces a \CuMor{}; see  \cite{WinZac09CpOrd0} and \cite{CowEllIva08CuInv} respectively.

The reader is referred to \cite{AntPerThi18TensorProdCu} for an in-depth introduction to \CuSgp{s}.
\end{pgr}


\section{\texorpdfstring{Pure \stHom{s}}{Pure *-homomorphisms}}\label{sc:PuereStHom}

We introduce in \cref{dfn:ZMultPure} a notion of \emph{$\Cu (\mathcal{Z})$-multiplication} for \CuMor{s}, and we say that a \stHom{} is \emph{pure} if its induced \CuMor{} has $\Cu (\mathcal{Z})$-multiplication. We then provide a  number of examples (\ref{exa:ZMult}-\ref{prp:ExaPureSR1}), and show that \textquoteleft{}$\mathcal{Z}$-stable like\textquoteright{} morphisms are always pure. In fact, their composition with the embedding to the codomain ultraproduct always factorizes through a pure \ca{}; see \cref{prp:ZstabImpPure}.

The section ends with \cref{qst:Pure1} and its weakening \cref{qst:Pure2}. A satisfactory general answer to the second question is given in \cref{subsec:McDuffness}.

\begin{pgr}\label{pgr:PurenessCuSgp}
Recall from \cite[Definition~2.1]{Win12NuclDimZstable} (see \cite{RobTik17NucDimNonSimple} for the concrete definition displayed here) that a \CuSgp{} $S$ is said to be
\begin{itemize}
\item \emph{almost unperforated} if, whenever $x,y\in S$ are such that $(m+1)x\leq my$ for some $m\in\NN$, one has $x\leq y$.
\item \emph{almost divisible} if for every $k\in\NN$ and $x',x\in S$ such that $x'\ll x$ there exists $z\in S$ such that $kz\leq x$ and $x' \leq (k+1)z$.
\end{itemize}

One says that a \ca{} is \emph{pure} if its Cuntz semigroup is almost unperforated and almost divisible.

In \cite{AntPerThi18TensorProdCu}, a theory of tensor products and \emph{multiplication} for Cuntz semigroups was developed. Amongst other results, the authors showed in \cite[Theorem~7.3.11]{AntPerThi18TensorProdCu} that a \CuSgp{} $S$ is pure if and only if $S\cong S\otimes \Cu (\mathcal{Z})$. Further, one can see that the Cuntz semigroup of any $\mathcal{Z}$-stable \ca{} has $\Cu (\mathcal{Z})$-multiplication (ie. it tensorially absorbs $\Cu (\mathcal{Z})$).
\end{pgr}

\begin{dfn}\label{dfn:ZMultPure}
 Let $\varphi\colon S\to T$ be a generalized \CuMor{}. We will say that $\varphi$ is
 \begin{itemize}
  \item[(i)] \emph{almost unperforated} if $\varphi (x)\leq \varphi (y)$ whenever $(m+1)x\leq my$ for some $m\in\mathbb N$.
 \item[(ii)] \emph{almost divisible} if for every $k\in\NN$ and $x',x\in S$ such that $x'\ll x$ there exists $z\in T$ such that $kz\leq \varphi (x)$ and $\varphi (x')\leq (k+1)z$.
 \end{itemize}
 
 The generalized \CuMor{} $\varphi$ will be said to have \emph{$\Cu (\mathcal{Z})$-multiplication} if it is almost unperforated and almost divisible, and a cpc. order-zero map $\theta\colon A\to B$ will be called \emph{pure} if $\Cu (\theta )$ has $\Cu (\mathcal{Z})$-multiplication.
\end{dfn}

Let us begin the section with some examples of pure maps:

\begin{exa}\label{exa:ZMult}
 It is readily checked that any \stHom{} $A\to B$ is pure whenever $A$ or $B$ is pure. More generally, the same holds if $\Cu (A\to B)$ factors through the Cuntz semigroup of a pure \ca{}; see \cref{prp:CompMult} for details and more general statements.
 
 As a concrete example, any \stHom{} that factors through an infinite reduced free product $\ast_{i=1}^{\infty} (A,\tau)$, with $\tau$ a trace, is pure. These products are always simple and monotracial (\cite[Proposition~3.1]{Avi82FreeProd}), have stable rank one (\cite[Theorem~3.8]{DykHaaRor97SRFreeProd}), and strict comparison (\cite[Proposition~6.3.2]{Rob12Class}). A standard argument, using the positive solution to the ranks problem for simple, stable rank one \ca{s} from \cite{Thi20RksOps}, proves that the Cuntz semigroup is almost divisible; see, for example, \cite[Remark~4.4]{Vil23NScaMult} or its generalization \cref{prp:ExaPureSR1} below.
\end{exa}

\begin{exa}
Let $A,B$ be \ca{s}, and assume that $B$ is almost divisible. Then, it follows from \cite[Lemma~6.1(i)]{RobRor13Divisibility} that the first factor embedding $A\to A\otimes B$ is almost unperforated and, therefore, pure.
\end{exa}

Another example of pure maps is given by the following proposition. Recall that a \ca{} $A$ is said to be \emph{nowhere scattered} if it has no nonzero elementary ideal-quotients; see \cite{ThiVil22NowSca}.

\begin{prp}\label{prp:ExaPureSR1}
Let $A$ be a nowhere scattered \ca{} of stable rank one, and let $\theta\colon A\to B$ be a cpc. order-zero, almost unperforated map. Then, $\theta$ is pure.
\end{prp}
\begin{proof}
We need to show that $\Cu (\theta)$ is almost divisible. To see this, let $x\in \Cu (A)$ and take $k\in\NN$. Then, it follows from \cite[Theorem~7.14]{AntPerRobThi22CuntzSR1} that there exists $y\in\Cu (A)$ such that $\widehat{y}=\frac{2}{2k+1}\widehat{x}$ in $L(F(\Cu (A)))$; see \cite[Section~7]{AntPerRobThi22CuntzSR1} for the appropiate definitions.

Thus, we have
\[
	\infty x=\infty y ,\quad
	\lambda (ky) < \lambda (x) ,\andSep
	\lambda (x) < \lambda ((k+1)y)
\]
for every normalized functional $\lambda \in F(\Cu (A))$.

By \cite[Proposition~2.1]{OrtPerRor12CoronaStability}, there exists $n\in\NN$ such that
\[
	(n+1)(ky)\leq nx,\andSep
	(n+1)x\leq n(k+1)y.
\]

Using that $\Cu (\theta )$ is almost unperforated, we obtain $k\Cu (\theta )(y)\leq \Cu (\theta )(x)\leq (k+1)\Cu (\theta )(y)$, as desired.
\end{proof}


Recall that a separable unital \ca{} is $\mathcal{Z}$-stable if and only if $\mathcal{Z}$ embeds unitally into $A_\omega \cap A'$; see, for example, \cite[Proposition 4.4]{Kir06CentralSeqPI}. A reindexing argument shows that this is equivalent to $\mathcal{Z}$ embedding unitally to $A_\omega \cap A'\cap C'$ for any separable sub-\ca{} $C\subseteq A_\omega \cap A'$. In particular, any such \ca{} satisfies that the unit $1\in A_\omega \cap A'\cap C'$ is almost divisible.

In what follows (\cref{prp:ZstabImpPure}), we show that any \textquoteleft{}$\mathcal{Z}$-stable-like\textquoteright{} \stHom{} is pure. For that, recall that given a \stHom{} $\theta\colon A\to B$, the annihilator of $\theta$, in symbols ${\rm Ann} (\theta)$, is defined as the set of  elements $x$ in $B$ such that $x\theta(A)=\theta(A)x=\{0\}$.

\begin{prp}\label{prp:RelCommUnit}
Let $\theta\colon A\to B$ be a \stHom{}. Assume that there exists a net of contractive, positive elements $(e_\lambda )_{\lambda\in\Lambda}$ in $\theta (A)'\cap B$ such that $(e_\lambda )_\lambda$ is an approximate unit for $\theta (A)$ and such that $\overline{e_\lambda}$ is almost divisible in $\theta (A)'\cap B/{\rm Ann}(\theta (A))$ for each $\lambda$. Then, there exists a sub-\ca{} $C\subseteq \theta (A)'\cap B$ such that the map $A\to C^*(\theta (A),C)$ is pure. If $\Lambda=\NN$, $C$ can be taken to be separable.
\end{prp}
\begin{proof}
For each triple $\mu=(\lambda , k ,p)\in\Lambda\times\NN\times\QQ_{+,\leq 1}$, let $f_\mu \in (\theta (A)'\cap B)_+$ be such that
\[
	\overline{f_\mu^{\oplus k}}\precsim \overline{e_\lambda},\andSep 
	\overline{(e_\lambda-p)_+}\precsim \overline{f_\mu^{\oplus (k+1)}}
\]
in $\Cu (\theta (A)'\cap B/{\rm Ann}(\theta (A)))$, where the superscript $\oplus k$ denotes the diagonal formed by $k$ copies of the element.

Thus, up to a cut-down of $f_\mu$, there exists a finite matrix $h_\mu$ with entries in ${\rm Ann}(\theta (A)))$ such that 
\[
	f_\mu^{\oplus k}\oplus 0\precsim e_\lambda\oplus h_\mu,\andSep 
	(e_\lambda-p)_+\oplus 0\precsim f_\mu^{\oplus (k+1)}\oplus h_\mu
\]
in $(\theta (A)'\cap B)\otimes M_l$ for some $l$, where the $0$'s denote zero matrices of appropiate sizes.

Now, for any $q\in\QQ_{+,\leq 1}$, it follows from \cite[Proposition~2.4]{Ror92StructureUHF2} that there exist finite matrices   $r_{\mu ,q}$ and $s_{\mu ,q}$ over $\theta (A)'\cap B$ such that 
\begin{equation}\label{eq:1}
	(f_\mu^{\oplus k}-q)_+\oplus 0 = r_{\mu ,q}(e_\lambda\oplus h_\mu )r_{\mu ,q}^*
\end{equation}
and
\begin{equation}\label{eq:2}
	(e_\lambda-p-q)_+\oplus 0 = s_{\mu ,q}(f_\mu^{\oplus (k+1)}\oplus h_\mu)s_{\mu ,q}^*.
\end{equation}

Set $C=C^*(\{ e_\lambda ,f_\mu, r_{\mu ,q}(i,j), s_{\mu ,q}(i,j)\}_{\mu ,q,i,j})$. If $\Lambda=\NN$, then $C$ is separable by construction.

Let us show that $A\to C^*(\theta (A),C)$ is pure. For any $m\in\NN$ and \ca{} $E$, let $\iota_m\colon E\to M_m (E)$ denote the map $e\mapsto e^{\oplus m}$. We will also use this notation for the matrix amplification $M_k (E)\to M_{km}(E)$ defined entry-wise. Note that, for a matrix $e\in M_n (E)$, $e^{\oplus m}$ denotes the $mn\times mn$ matrix $e\oplus \ldots \oplus e$, while $\iota_m (e)$ is the $mn\times mn$ where each entry of $e$ has been replaced by a diagonal $m\times m$ matrix. An important fact that we will use is: Given $a\in M_n (A)$ and $b\in M_m (\theta (A)'\cap B)$, then $\iota_n (b)$ commutes with $\theta (a)^{\oplus m}$ (and the roles of $a$ and $b$ can be reversed).

First, given a contraction $a\in M_n (A)_+$, note that applying $\iota_n$ to \cref{eq:1} and multiplying by an $l=l(\mu ,q)$ diagonal matrix of $\theta(a)$'s, one gets
\[
\begin{split}
	(\theta(a)(f_\mu-q)_+^{\oplus n})^{\oplus k}\oplus 0 &= \theta(a)^{\oplus l} \iota_n((f_\mu^{\oplus k}-q)_+\oplus 0)\\ &= \theta(a)^{\oplus l}\iota_n(r_{\mu ,q})\iota_n(e_\lambda\oplus h_\mu )\iota_n(r_{\mu ,q})^*\\
&= \iota_n (r_{\mu ,q}) \theta(a)^{\oplus l} \iota_n (e_\lambda\oplus h_\mu ) \iota_n(r_{\mu ,q})^*\\ &= \iota_n (r_{\mu ,q}) (\theta(a) \iota_n(e_\lambda ) \oplus 0)\iota_n(r_{\mu ,q})^*\\
&\precsim \theta (a)\iota_n (e_\lambda )
\end{split}
\]
and, by multiplying an $m=m(\mu ,q)$ diagonal matrix of $a$'s to an enlarged \cref{eq:2}, we obtain
\[
\begin{split}
\theta(a)(e_\lambda-p-q)_+^{\oplus n}\oplus 0 &= \theta(a)^{\oplus m}\iota_n(s_{\mu ,q}) \iota_n (f_\mu^{\oplus (k+1)}\oplus h_\mu) \iota_n(s_{\mu ,q})^*\\
&= \iota_n (s_{\mu ,q})  ((\theta(a)\iota_n f_\mu^{\oplus (k+1)}\oplus 0) \iota_n(s_{\mu ,q})^*\\
&= \iota_n (s_{\mu ,q})  ((\theta(a) f_\mu^{\oplus n})^{\oplus (k+1)}\oplus 0) \iota_n(s_{\mu ,q})^*\\
&\precsim (\theta (a)f_\mu^{\oplus n})^{\oplus (k+1)}
.
\end{split}
\]

Thus,
\[
((\theta(a) f_\mu^{\oplus n}-q)_+)^{\oplus k}\precsim \theta(a) e_\lambda^{\oplus n}\precsim \theta (a),\andSep
(\theta(a)( e_\lambda^{\oplus n}-p-q)_+)\precsim (\theta(a) f_\mu^{\oplus n})^{\oplus (k+1)}
\]
in $C^*(\theta (A),C)\otimes\mathcal{K}$.

Since this holds for any choice of $p$ and $q$, taking suprema gives $(\theta(a) f_\mu^{\oplus n})^{\oplus k}\precsim \theta (a)$ and $(\theta(a) e_\lambda^{\oplus n})\precsim (\theta(a) f_\mu^{\oplus n})^{\oplus (k+1)}$. Further, since the $e_\lambda^{\oplus n}$'s are an approximate unit for $\theta(a)$, it follows that $a$ is almost divisible in $C^*(\theta (A),C)\otimes\mathcal{K}$. Since this holds for any finite matrix over $\theta (A)$, we get that  $A\to C^*(\theta (A),C)$ is almost divisible.

To show that $A\to C^*(\theta (A),C)$ is almost unperforated, we follow the same strategy as in \cite[Lemma~4.3]{Ror04StableRealRankZ}. Thus, let $a,b\in M_n (A)_+$ be contractions such that $a^{\oplus (k+1)}\precsim b^{\oplus k}$ in $A\otimes \mathcal{K}$ for some $k\in\NN$. For any $\varepsilon >0$, there exists a finite matrix $v$ with entries in $A$ such that $(a-\varepsilon )_+^{\oplus (k+1)}= v (b^{\oplus k}\oplus 0) v^*$. Given $\mu = (\lambda ,k,p)$, we have 
\[
	(f_\mu^{\oplus n}\theta((a-\varepsilon )_+))^{\oplus (k+1)}
	= \theta (v) ((f_\mu^{\oplus n}\theta (b))^{\oplus k}\oplus 0) \theta (v)^*.
\]

In particular, $(f_\mu^{\oplus n}\theta((a-\varepsilon )_+))^{\oplus (k+1)} \precsim (f_\mu^{\oplus n}\theta (b))^{\oplus k}$ in $\Cu (C^*(C,A))$. Recall from the above computations that we have
\[
(\theta((a-\varepsilon )_+) e_\lambda^{\oplus n})\precsim (\theta((a-\varepsilon )_+) f_\mu^{\oplus n})^{\oplus (k+1)}
,
\andSep
(\theta(b) f_\mu^{\oplus n})^{\oplus k}\precsim \theta (b)
.
\]

Chaining  these three $\precsim$-relations together, one obtains
\[
(\theta((a-\varepsilon )_+) e_\lambda^{\oplus n})\precsim (\theta((a-\varepsilon )_+) f_\mu^{\oplus n})^{\oplus (k+1)}\precsim (f_\mu^{\oplus n}\theta (b))^{\oplus k}
\precsim \theta (b).
\]

Using once again that  the $e_\lambda^{\oplus n}$'s are an approximate unit for $\theta((a-\varepsilon )_+)$, we get $(\theta (a)-2\varepsilon )_+\precsim\theta (b)$. This proves that $A\to C^*(\theta (A),C)$ is almost unperforated, as required.
\end{proof}

\cref{prp:PureUltra} below shows that, when studying pureness of maps, one can restrict to those with an ultraproduct for a codomain.

\begin{lma}\label{prp:PureUltra}
Let $\theta\colon A\to B$ be a cpc. order-zero map, and let $\iota_\omega \colon B\to B_\omega$ be the natural inclusion. Then, $\theta$ is pure if and only if $\iota_\omega\theta$ is pure.
\end{lma}
\begin{proof}
If $\theta$ is pure, the composition $\iota_\omega\theta$ is pure; see \cref{prp:CompMult} for details.

Conversely, assume that $\iota_\omega\theta$ is pure. We need to show that $\Cu (\theta)$ is both almost divisible and almost unperforated. First, let $n\in\NN$, $[a]\in\Cu (A)$ and $\varepsilon >0$. Since any element in $\Cu (A)$ can be written as the supremum of classes in $M_\infty (A)_+$ (and since $\Cu (\theta )$ preserves suprema), we may assume $a\in M_k (A)_+$ for some $k\in\NN$. Using that $\iota_\omega\theta$ is pure, we find $[b]\in\Cu (B_\omega)$ such that $[\iota_\omega\theta ((a-\varepsilon /3)_+)]\leq (n+1)[b]$ and $n[b]\leq [\iota_\omega\theta (a)]$. Upon cutting-down $a$ and $b$ if needed (say, to $(a-\varepsilon /2)_+$), we may assume $b\in B_\omega\otimes M_k$ and that there exist finite matrices $r,s$ over $B_\omega$ such that $\iota_\omega\theta ((a-\varepsilon /2)_+)=rb^{\oplus (n+1)}r^*$ and $b^{\oplus n}=s\iota_\omega\theta (a)s^*$.

Further, note that $B_\omega \otimes M_k\cong (B \otimes M_k)_\omega$. Thus, by going sufficiently far in the sequences corresponding to $b$ and the entries of $r$ and $s$, we can  find elements $b_0\in M_k (B)$ and $r_0,s_0\in M_\infty (B)$ such that $[\theta ((a-\varepsilon )_+)]\leq (n+1) [b_0]$ and $n[b_0]\leq[\theta (a)]$. This shows that $\Cu (\theta )$ is almost divisible.

The proof of almost unperforation is analoguous.
\end{proof}

\begin{thm}\label{prp:ZstabImpPure}
Let $A$ be $\sigma$-unital, and let $\theta\colon A\to B$ be a \stHom{}. Assume that $1\in B_\omega\cap\theta (A)'/{\rm Ann}(\theta (A))$ is almost divisible. Then, $\theta$ is pure.
 
 Further, if $1\in B_\omega\cap(\theta (A)\cup S)'/{\rm Ann}(\theta (A)\cup S)$ is almost divisible for every separable sub-\ca{} $S\subseteq B_\omega\cap\theta (A)'$, there exists a pure sub-\ca{} $C\subseteq B_\omega$ such that $\theta (A)\subseteq C \subseteq B_\omega$.
\end{thm}
\begin{proof}
As shown in \cref{prp:PureUltra} above, a \stHom{} $A\to B$ is pure if and only if the induced map $A\to B_\omega$ is pure. Further, since $1\in B_\omega\cap\theta (A)'/{\rm Ann}(\theta (A))$ is almost divisible, it follows from \cref{prp:RelCommUnit} above that $A\to B_\omega$ is pure. This gives the first part of the statement.

For the second part, note that it follows from \cref{prp:RelCommUnit} that there exists a separable sub-\ca{} $C_1\subseteq B_\omega\cap \theta (A)'$ such that inclusion $\theta (A)\subseteq C^*(\theta (A), C_1)$ is pure. Since $C_1$ is separable, we know by our assumption that $1\in B_\omega\cap(\theta (A)\cup C_1)'/{\rm Ann}(\theta (A)\cup C_1)$ is almost divisible. \cref{prp:RelCommUnit} proves the existence of a separable $C_2\subseteq B_\omega\cap\theta (A)'\cap C_1'$ such that $C^*(\theta (A), C_1)\subseteq C^*(\theta (A), C_1, C_2)$ is pure.

Proceeding inductively, one obtains a sequence of pure inclusions in $B_\omega$. Their limit, denoted by $C$, is pure.
 \end{proof}

 The previous result justifies the following question:
 
 \begin{qst}\label{qst:Pure1}
  Let $\theta\colon A\to B$ be a \stHom{}. Is $\theta$ pure if and only if the composition $\iota_w\theta\colon A\to B_\omega$ factors, up to Murray-von Neumann equivalence (see \cref{prp:ApproxMvNMult}), through a pure \ca{}?
 \end{qst}
 
 \cref{qst:Pure1} is a particular instance of a general question that one can ask for any property $P$. Namely, does any \stHom{} with $P$ \textquoteleft{}come\textquoteright{} from a \ca{} with $P$? In other words, does $P$ admit a McDuff type characterization?
 
 This question has been posed for: real rank zero inclusions \cite{GabNea23RRZero}, where it remains open; for $\mathcal{O}_2$-stable morphisms, answered in \cite[Corollary~4.5]{Gabe2020ANewProofKirchberg}; and
for morphisms of nuclear dimension 0, with a partial answer provided in \cite{CASTILLEJOS2024110368}.
 
 For any of the properties $P$ listed above, one has that an inductive system where each map satisfies $P$ has a limit with $P$. Loosely, one can interpret this as saying that any infinite composition of maps with $P$ always factorizes through a \ca{} with $P$. In this sense, one may also ask if there exists a natural number $n_P$ such that the composition of $n_P$ morphisms with $P$ always factorizes through a \ca{} with $P$. Specialized to our setting, the question is:

  \begin{qst}\label{qst:Pure2}
  Does there exist $n\in\NN$ such that, for any tuple $\theta_1,\ldots ,\theta_n$ of pairwise composable pure \stHom{s}, the composition $\iota_w\theta_n\cdots \theta_1$ factors up to Murray-von Neumann equivalence through a pure \ca{}?
 \end{qst}
 
 To our knowledge, an answer to \cref{qst:Pure2} is not known for real rank zero inclusions. In what follows, we investigate the question for our notion of pureness.


\section{\texorpdfstring{Pureness and $\Cu (\mathcal{Z})$-multiplication}{Pureness and Cu(Z)-multiplication}}\label{subsec:PermanenceProp}

This section compiles permanence properties of $\Cu (\mathcal{Z})$-multiplication for gene\-ra\-li\-zed \CuMor{s}. We state some of these results in the language of abstract Cuntz semigroups to highlight when the $\ll$-relation needs to be preserved.

Propositions \ref{prp:idMult} and \ref{prp:CompMult} below are in analogy to Proposition~3.19 and Lemma~3.20 from \cite{Gabe2020ANewProofKirchberg}.

\begin{prp}\label{prp:idMult}
 Let $S$ be a \CuSgp{}. Then $S\cong S\otimes \Cu (\mathcal{Z})$ if and only if $\id_S$ has $\Cu (\mathcal{Z})$-multiplication.
\end{prp}
\begin{proof}
 It follows from Theorems~7.3.11 and 7.5.4 in \cite{AntPerThi18TensorProdCu} that $S\cong S\otimes \Cu (\mathcal{Z})$ if and only if $S$ is almost divisible and almost unperforated. The statement now follows from the defintions.
\end{proof}


\begin{prp}\label{prp:CompMult}
 Let $\varphi_1\colon S_1\to S_2$ and $\varphi_2\colon S_2\to T$ be generalized \CuMor{s}. Then,
 \begin{enumerate}
  \item if $\varphi_1$ has $\Cu (\mathcal{Z})$-multiplication, so has $\varphi_2\varphi_1$.
  \item if $\varphi_1$ is a \CuMor{} and $\varphi_2$ has $\Cu (\mathcal{Z})$-multiplication, the composition $\varphi_2\varphi_1$ also has $\Cu (\mathcal{Z})$-multiplication.
 \item if $S_2$ has $\Cu (\mathcal{Z})$-multiplication, so have both $\varphi_2$ and $\varphi_1$.
 \end{enumerate}
\end{prp}
\begin{proof}
 (1) Given $x'\ll x$ in $S_1$ and $k\in\NN$, there exists $z\in S_2$ such that $\varphi_1 (x')\leq (k+1)z$ and $kz\leq \varphi_1 (x)$. Thus, one has $\varphi_2\varphi_1 (x')\leq (k+1)\varphi_2 (z)$ and $k\varphi_2 (z)\leq \varphi_2\varphi_1 (x)$.
 
 Similarly, if $(m+1)x\leq my$ for some $m\in\NN$ in $S_1$, one gets $\varphi_1 (x)\leq \varphi_1 (y)$ and, consequently, $\varphi_2\varphi_1 (x)\leq \varphi_2\varphi_1 (y)$.
 
 (2) If $\varphi_1$ is now a \CuMor{} and $\varphi_2$ is the map that has $\Cu (\mathcal{Z})$-multiplication, one can take $x'\ll x$ in $S_1$ and consider the induced relation $\varphi_1 (x')\ll\varphi_1 (x)$ in $S_2$. Since $\varphi_2$ has $\Cu (\mathcal{Z})$-multiplication, one obtains the element needed for the almost division in $T$.
 
 The same argument shows that if $(m+1)x\leq my$ in $S$, then $\varphi_2\varphi_1 (x)\leq \varphi_2\varphi_1 (y)$ in $T$, as desired.

(3) If $S_2$ has $\Cu (\mathcal{Z})$-multiplication, \cite[Proposition~7.3.8]{AntPerThi18TensorProdCu} implies that for any element $x\in S_2$ and $k\in \mathbb N$, there exists $y\in S_2$ such that $ky\leq x\leq (k+1)y$. With this stronger property, it is routine to check that both $\varphi_1$ and $\varphi_2$ have $\Cu (\mathcal{Z})$-multiplication. 
\end{proof}

The following lemma follows directly from standard model theoretic techniques applied to \CuSgp{s}; see, for example, \cite[Section~5]{ThiVil21DimCu2}.

\begin{lma}\label{prp:LSkMult}
 Let $\varphi\colon S\to T$ be a generalized \CuMor{}. Then $\varphi$ has $\Cu (\mathcal{Z})$-multiplication if and only if $\varphi\iota_{H}$ has $\Cu (\mathcal{Z})$-multiplication for every inclusion $\iota_H$ from a countably based sub-\CuSgp{} $H$ to $S$.
\end{lma}


Recall from \cite[Definition~3.4]{Gabe2020ANewProofKirchberg} that a pair of  \stHom{s} $\theta ,\eta\colon A\to B$ are \emph{approximately Murray-von Neumann equivalent} if, for any finite subset $\mathcal{F}$ of $A$ and any $\varepsilon >0$, there exists $u\in B$ such that 
\[
 \Vert u\theta (a) u^* -\eta (a) \Vert <\varepsilon, \andSep
  \Vert u\eta (a) u^* -\theta(a) \Vert <\varepsilon,
\]
for every $a\in\mathcal{F}$.

The following is essentially \cite[Corollary~3.11]{Gabe2020ANewProofKirchberg} applied to the category $\Cu$, with the only difference that we drop the assumption of $A$ being separable.

\begin{lma}\label{prp:ApproxMvNMult}
 Let $\theta ,\, \eta\colon A\to B$ be two approximately Murray-von Neumann equivalent \stHom{s}. Then, $\Cu (\theta )=\Cu (\eta )$.
\end{lma}
\begin{proof} 
Let $x\in \Cu (A)$. By \cite[Proposition~6.1]{ThiVil21DimCu2}, there exists a separable sub-\ca{} $A'\subseteq A$ such that $\Cu (\iota_{A'})(\Cu (A'))$ is a sub-\CuSgp{} of $\Cu (A)$ containing $x$. Here, $\iota_{A'}$ denotes the inclusion from $A'$ to $A$.

Since the functor $\Cu (\cdot )$ is $M_2$-stable and invariant under approximate unitary equivalence, \cite[Corollary~3.11]{Gabe2020ANewProofKirchberg} implies that 
\[
 \Cu (\theta)(x) = \Cu (\theta \iota_{A'})(x)=\Cu (\eta \iota_{A'})(x) = \Cu (\eta )(x),
\]
as required.
\end{proof}

The following result gives one of the implications of \cref{qst:Pure1}.

\begin{prp}\label{prp:OinftyMultGabe}
 Let $\theta\colon A\to B$ be a \stHom{}. Assume that, for any separable sub-\ca{} $A'\subseteq A$, there exist cpc. order-zero maps $\eta\colon A'\to C$ and $\rho\colon C\to B$ such that $C$ is pure and $\rho\eta$ is approximately Murray-von Neumann equivalent to $\theta\iota_{A'}$.
 
 Then, $\theta$ is pure.
\end{prp}
\begin{proof}
 Assume first that $A$ is separable. Then, $\theta$ is Murray-von Neumann equivalent to a map that factorizes through $C$. Thus, by \cref{prp:ApproxMvNMult}, $\Cu (\theta )$ itself factorizes through $\Cu (C)$, and so $\theta$ is pure by \cref{prp:CompMult}.
 
 If $A$ is not separable, we know from \cite[Proposition~6.1]{ThiVil21DimCu2} that any countably based sub-\CuSgp{} $H$ in $\Cu (A)$ is contained in $\Cu (A')$ for some separable sub-\ca{} $A'$ of $A$. Thus, the  argument above shows that $\Cu (\theta \iota_{A'})$ has $\Cu (\mathcal{Z})$-multiplication.
 
 Consequently, since any inclusion gives rise to a \CuMor{},  \cref{prp:CompMult} shows that $\Cu (\varphi)\iota_H$ has $\Cu (\mathcal{Z})$-multiplication. \cref{prp:LSkMult} gives the required result.
\end{proof}

\begin{pgr}[Approximations of \stHom{s}]
Recall that a \ca{} $A$ is said to be \emph{approximated} by a family of sub-\ca{s}
$(A_\lambda )_{\lambda\in\Lambda}$ if, for each $\varepsilon > 0$ and every choice of finitely many elements $a_1,\ldots ,a_n\in A$, there exist $\lambda\in\Lambda$ and $b_1,\ldots ,b_n\in A_\lambda$ such that $\Vert b_j-a_j\Vert <\varepsilon$ for each $j$.

Let $\theta\colon A\to B$ be a cpc. order-zero map. We will say that a tuple $(A_\lambda , \theta_\lambda\colon A_\lambda\to B )_{\lambda\in\Lambda}$ \emph{approximates} $\theta$ if each $\theta_\lambda$ is cpc. order-zero and the following condition holds:

For every $\varepsilon > 0$ and every finitely tuple $a_1,\ldots ,a_n\in A$, there exist $\lambda\in\Lambda$ and $b_1,\ldots ,b_n\in A_\lambda$ such that
\[
	\Vert b_j-a_j\Vert <\varepsilon ,\andSep
	\Vert \theta_\lambda (b_j)-\theta (a_j)\Vert <\varepsilon
\]
for each $j$.

Note that this notion of approximation naturally includes the notion of limit morphism (ie. when $A_\lambda = A$ for each $\lambda$).
\end{pgr}

We will now show that pureness is preserved under approximations. We do this by proving a much more general result, which we expect to find other uses elsewhere. Informally, \cref{prp:ApproxCuMorph} below says that any formula of the Cuntz semigroup is inherited by the approximated map. This generalizes \cite[Proposition~3.7]{ThiVil21DimCu2}.

\begin{prp}\label{prp:ApproxCuMorph}
Let $\theta\colon A\to B$ be a cpc. order-zero map, and let $(A_\lambda , \theta_\lambda )_{\lambda\in\Lambda}$ approximate $\theta$. Then, for any finite sets $J,K$, any family of pairs $[a_j'],[a_j]\in \Cu (A)$ such that $[a_j']\ll [a_j]$ for each $j\in J$, and any functions $m_k,n_k\colon J\to \NN$ such that 
\[
	\sum_{j \in J}m_k (j) [a_j] \ll \sum_{j\in J} n_k (j)  [a_j']
\]
for all $k\in K$, there exists $\lambda\in \Lambda$, and $c_j\in (A_\lambda\otimes\mathcal{K})_\lambda$ for each $j$, such that $[\theta (a_j' )]\ll [\theta_\lambda (c_j)]\ll [\theta (a_j)]$, and $[a_j' ]\ll [c_j]\ll [a_j]$ in $\Cu (A)$, and 
\[
	\sum_{j \in J}m_k (j) [c_j] \ll \sum_{j\in J} n_k (j)  [c_j]
\]
in $\Cu (A_\lambda)$.
\end{prp}
\begin{proof}
Let $\varepsilon >0$ be such that $[a_j']\leq [(a_j-2\varepsilon )_+]$ for each $j$. Note that, by definition, the $A_\lambda$'s approximate $A$. Thus, it follows from (the proof of)  \cite[Proposition~3.7]{ThiVil21DimCu2} that, for every sufficiently small positive $\sigma>0$ with $\sigma<\varepsilon$, one can find $\lambda\in \Lambda$ and $b_j\in (A_\lambda\otimes\mathcal{K})_\lambda$ such that $[a_j' ]\ll [(b_j-\varepsilon )_+]\ll [a_j]$ in $\Cu (A)$ and 
\[
	\sum_{j \in J}m_k (j) [(b_j-\varepsilon )_+] \ll \sum_{j\in J} n_k (j)  [(b_j-\varepsilon )_+]
\]
in $\Cu (A_\lambda)$ for every $k$ and, additionally, such that $\Vert \theta_\lambda (b_j)-\theta (a_j)\Vert\leq\sigma$ for every $j$.

Set $c_j:= (b_j-\varepsilon )_+$. Then, since $\sigma <\varepsilon$, it follows that $\theta ((a_j-2\varepsilon )_+)\precsim \theta_\lambda (c_j)\precsim \theta (a_j)$. Using that $\theta$ is cpc. order zero, we see that $\theta (a_j')\precsim \theta ((a_j-2\varepsilon )_+)$ and, consequently, $\theta (a_j')\precsim \theta_\lambda (c_j)$.
\end{proof}

\begin{cor}
Let $\theta\colon A\to B$ be a cpc. order-zero map, and let $(A_\lambda , \theta_\lambda )_{\lambda\in\Lambda}$ approximate $\theta$. Assume that $\theta_\lambda$ is pure for each $\lambda$. Then, $\theta$ is pure.
\end{cor}
\begin{proof}
First, let $n\in\NN$. Given $[a]\in\Cu (A)$ and $\varepsilon >0$, use \cref{prp:ApproxCuMorph} (with $K=\emptyset $) to find $\lambda\in\Lambda$ and $c\in (A_\lambda\otimes\mathcal{K})_+$ such that $[\theta ((a-\varepsilon)_+ )]\ll [\theta_\lambda (c)]\ll [\theta (a)]$ and $[(a-\varepsilon)_+ ]\ll [c]\ll [a]$. Take $\delta >0$ such that $[\theta((a-\varepsilon)_+) ]\leq [\theta_\lambda ((c-\delta )_+)]$. Then, since $\theta_\lambda$ is pure, there exists $[d]\in\Cu (B)$ such that $n[d]\leq [\theta_\lambda (c)]$ and $[\theta_\lambda((c-\delta )_+)]\leq (n+1)[d]$ in $\Cu (B)$. This implies
\[
n[d]\leq [\theta_\lambda (c)]\leq [\theta (a)]
,\andSep 
[\theta ((a-\varepsilon)_+)] \leq [\theta_\lambda((c-\delta )_+)]\leq (n+1)[d]
,
\]
which shows that $\theta$ is almost divisible.

Now assume that $[a_1],[a_2]\in \Cu (A)$ are such that $(m+1)[a_1]\leq m[a_2]$ for some $m\in\NN$. Take any pair of elements $a_1', a_1''$ such that $[a_1']\ll [a_1'']\ll [a_1]$, and find $a_2'$ such that $[a_2']\ll [a_2]$ and $(m+1)[a_1'']\ll m[a_2']$. Apply \cref{prp:ApproxCuMorph} for the pairs $a_1',a_1''$ and $a_2',a_2$ and the formula $(m+1)[a_1'']\ll m[a_2']$ to find $\lambda\in\Lambda$ and $c_1,c_2\in (A_\lambda\otimes\mathcal{K})_+$ such that
\[
	[\theta (a_1')]\ll [\theta_\lambda (c_1)] \ll [\theta (a_1'')],\quad 
	[\theta (a_2')]\ll [\theta_\lambda (c_2)] \ll [\theta (a_2)]
\]
in $\Cu (A)$ and $(m+1)[c_1]\ll m[c_2]$ in $\Cu (A_\lambda)$.

Thus, since $\Cu (\theta_\lambda )$ is almost unperforated, $[\theta_\lambda (c_1)]\leq [\theta_\lambda (c_2)]$. This shows that $[\theta (a_1')]\leq [\theta (a_2)]$ and, since the choice of $a_1'$ was arbitrary, we obtain $[\theta (a_1)]\leq [\theta (a_2)]$.
\end{proof}


\section{\texorpdfstring{Factorizing compositions of pure \stHom{s}}{Factorizing compositions of pure *-homomorphisms}}\label{subsec:McDuffness}

The aim of this section is to provide a partial answer to \cref{qst:Pure2}. Namely, we show that ---at a Cuntz semigroup level--- the composition of any two pure \stHom{s} (in fact, two cpc. order-zero maps) factors through a pure \CuSgp{}; see \cref{thm:MainCuA}. In order to prove such a result, we start with a study of pureness in the category $\Cu$, which we use to prove the technical result \cref{prp:McDuffZ}.

Recall that the Cuntz semigroup of the Jiang-Su algebra $\mathcal{Z}$ is  isomorphic to $Z=\NN \cup (0,\infty]$(see eg. \cite{AntPerThi18TensorProdCu}).  Throughout this subsection, we will write $Z$ as the union $Z=Z_c\cup Z_\soft$, where $Z_c=\NN$ and $Z_\soft = [0,\infty ]$ with $Z_c\cap Z_\soft =\{ 0\}$. Denote by $\sigma\colon Z\to [0,\infty ]$ the soft retraction, that is, the map that sends each compact element $n\in Z_c$ to its soft counterpart $n\in Z_{\soft }=[0,\infty ]$ and leaves the soft part invariant.

\begin{lma}\label{prp:ExtensionZtoW}
Let $S$ be a \CuSgp{}, and let $\gamma\colon Z\to S$ be a map such that $\gamma|_{Z_c}$ is an order-preserving monoid morphism. Then, $\gamma$ is a generalized \CuMor{} if and only if the following two conditions are satisfied:
\begin{itemize}
 \item[(i)] $\gamma (\sigma (1))\leq \gamma (1)\leq \gamma (1+\varepsilon )$ for every $\varepsilon >0$; and
 \item[(ii)] $\gamma|_{Z_\soft}$ is a generalized \CuMor{}.
\end{itemize}
\end{lma}
\begin{proof}
The forward implication is trivial. For the reverse implication, assume that $\gamma|_{Z_\soft}$ is a generalized \CuMor{} and that $\gamma (\sigma (1))\leq \gamma (1)\leq \gamma (1+\varepsilon )$ for each $\varepsilon >0$. Note, in particular, that we have $\gamma (\sigma (n))\leq \gamma (n)\leq \gamma (n+\varepsilon )$ for each $n\in Z_c$. To show that $\gamma$ is order-preserving, take $n\in Z_c$ and $t\in Z_\soft$ such that $t\leq n$. Then, one has $t\leq \sigma (n)$ and, consequently, $\gamma (t)\leq \gamma (\sigma (n))\leq \gamma (n)$. Conversely, if $n\leq t$, we know that $\sigma (n)<t$ in $Z_\soft$. Let $\varepsilon>0$ be such that $\sigma (n)+\varepsilon <t$. Then, 
\[
	\gamma (n)\leq \gamma (n+\varepsilon )=\gamma (\sigma (n)+\varepsilon )\leq \gamma (t).
\]

To show that it preserves suprema, note that any increasing sequence in $Z$ has a cofinal subsequence either in $Z_c$ or $Z_\soft$. Thus, we may assume that we are in one of these two cases. If the increasing sequence $(t_d)_d$ is in $Z_\soft = (0,\infty ]$, $\gamma$ preserves its supremum by assumption. Else, if $(t_d)_d$ is in $Z_c$, it either stabilizes (in which case $\gamma$ trivially preserves its supremum) or it tends to $\infty\in Z_\soft$. In this situation, one can take the sequence $(\sigma(t_d) )_d$ induced by the soft elements corresponding to our compact sequence. These two sequences share $\infty$ as their supremum. One has
\[
	\gamma (\sigma (t_d))\leq \gamma (t_d)\leq \sup_d \gamma (t_d),\andSep 
	\gamma (t_d)\leq \gamma \left( t_d+\frac{1}{d} \right)
\]
for every $d\geq 2$. This implies $\gamma (\infty )=\sup_d\gamma (\sigma (t_d))\leq \sup_d \gamma (t_d)$ and $\sup_d \gamma (t_d)\leq \sup_d\gamma (t_d+\frac{1}{d} )=\gamma (\infty )$. This shows $\sup_d \gamma (t_d)=\gamma (\infty )$, as desired.

Finally, to see that the map is additive, take $n\in Z_c$ and $t\in Z_\soft$. Then,
\[
	\gamma (n+t) = \gamma (\sigma (n)+t)=\gamma (\sigma (n))+\gamma (t)\leq \gamma (n)+\gamma (t).
\]

Conversely, if $t\neq 0$, let $\varepsilon >0$ such that $t-\varepsilon >0$. Then, $n+t=(\sigma(n)+\varepsilon)+(t-\varepsilon )$. This implies 
\[
	\gamma (n)+\gamma (t-\varepsilon )\leq \gamma (\sigma (n)+\varepsilon )+\gamma (t-\varepsilon )
	= \gamma (n+t )
\]
and, letting $\varepsilon$ tend to $0$, we obtain $\gamma (n)+\gamma (t)\leq \gamma (n+t )$, as required.
\end{proof}

Let $S$ be a \CuSgp{}. The following notation is inspired by \cite[Theorem~6.3.3]{AntPerThi18TensorProdCu}: For any pair $x'\leq x$ and any $k,n\in\NN$, set 
\[
	\mu ((k,n),x',x):=\{
		y\in S\mid ny\leq kx,\text{ and } kx'\leq (n+1)y
	\}.
\]
Note that this set is not empty whenever $x'\ll x$ and $x$ is almost divisible. Further, one has that
\[
	\mu ((k,n),x'',x)\subseteq \mu ((k,n),x',x)\subseteq \mu ((k,n),0,x)
\]
whenever $x'\leq x''\leq x$, and that $\mu ((k,n),0,x)=\{ y\in S\mid ny\leq kx \}$.

What follows is a generalization of \cite[Theorem~6.3.3]{AntPerThi18TensorProdCu} to our setting. In our case, the proof becomes more technical (for example, we cannot use unicity arguments) and so we proceed with additional care. As another difference between the methods, we will need to use the extension result (\cref{prp:ExtensionZtoW}) proved above.

\begin{lma}\label{lma:AlmUnpfMap}
Let $\varphi\colon S\to T$ be an almost unperforated generalized $\Cu$-morphism. Let $x_1,x_2\in S$ and $k_1,k_2,n_1,n_2\in\NN$ such that $k_1/n_1<k_2/(n_2+1)$. Assume that $x_1\leq x_2$. Then, 
\begin{enumerate}
\item $\varphi (y_1)\leq \varphi(y_2)$ for every $y_1\in \mu ((k_1,n_1),0,x_1)$ and $y_2\in \mu ((k_2,n_2),x_1,x_2)$.
\item If $\varphi$ is a \CuMor{}, then $\varphi (y_1)\ll \varphi(y_2)$ whenever $y_1\in \mu ((k_1,n_1),0,x_1)$, $y_2\in \mu ((k_2,n_2),x_2',x_2)$ and $x_1\ll x_2'\ll x_2$.
\end{enumerate}
\end{lma}
\begin{proof}
One has $n_1 y_1\leq k_1 x_1$ and $k_2 x_1\leq (n_2+1)y_2$. In particular,
\[
	n_1k_2 y_1\leq k_1k_2x_1\leq (n_2+1)k_1 y_2.
\]

It follows from definition that $\varphi (y_1)\leq \varphi(y_2)$, which shows (1).

For (2), simply note that one gets
\[
	n_1k_2 y_1\leq k_1k_2x_1\ll k_1k_2x_2' \leq (n_2+1)k_1 y_2.
\]

Thus, we can find $y_2'$ such that $y_2'\ll y_2$ and $n_1k_2 y_1\leq (n_2+1)k_1y_2'$. This implies $\varphi (y_1)\leq \varphi (y_2')\ll \varphi (y_2)$, as required.
\end{proof}

\begin{lma}\label{prp:SupZabso}
Let $\varphi_1\colon S_1\to S_2$ and $\varphi_2\colon S_2\to T$ be generalized $\Cu$-morphisms. Assume that $\varphi_1$ is almost divisible, and that $\varphi_2$ is almost unperforated. Then, for any $x\in S_1$ and $t\in (0,\infty ]$, the set
\[
	\Phi (t,\varphi_1(x)) := \left\{
		\varphi_2 (y) \mid y\in \mu ((k,n),0,\varphi_1 (x)) \text{ for some } k,n\in\NN \text{ such that } \frac{k}{n}<t
	\right\}
\]
has a supremum, bounded by $\lceil t \rceil \varphi_2\varphi_1 (x)$ (here, $\lceil \infty \rceil :=\infty$.)
\end{lma}
\begin{proof}
For every $d\in\NN$, take $k_d,n_d\in\NN$ and $x_d\in S_1$ such that
\[
	\frac{k_d}{n_d}<\frac{k_{d+1}}{n_{d+1}+1} ,\quad
	\sup_d \left(\frac{k_d}{n_d}\right)= t ,\quad
	x_d\ll x_{d+1},\andSep 
	\sup_d x_d = x
	.
\]

Set $x_0=0$. For each $d$, take $y_d\in \mu ((k_d,n_d),\varphi_1(x_{d-1}),\varphi_1(x_d))$, which exists by almost divisibility of $\varphi_1$. By \cref{lma:AlmUnpfMap}~(1), we see that the sequence $(\varphi_2 (y_d))_d$ is increasing in $T$. Consider $z=\sup_d \varphi_2 (y_d)$. We will prove that $z$ is the supremum of $\Phi (t,\varphi_1(x))$.

Take $y\in \mu ((k,n),0,\varphi_1 (x))$ for some $k,n\in\NN$ such that $k/n<t$. Take $y'\in S_2$ such that $y'\ll y$, and find $d\in\NN$ such that
\[
	\frac{k}{n}< \frac{k_{d+1}}{n_{d+1}+1},\andSep 
	y'\in \mu ((k,n),0,\varphi_1 (x_d)).
\]

Since $y_{d+1}\in \mu ((k_{d+1},n_{d+1}),\varphi_1 (x_d),\varphi_1 (x_{d+1}))$, it follows from \cref{lma:AlmUnpfMap} (1) that $\varphi_2 (y')\leq \varphi_2 (y_{d+1})\leq z$. As this holds for every $y'$ $\ll$-below $y$, we get $\varphi_2 (y)\leq z$. This shows that $z$ is the supremum of $\Phi (t,\varphi_1 (x))$, as desired.

To see that $z$ is bounded by $\lceil t \rceil \varphi_2\varphi_1 (x)$, simply note that for any pair $k,n$ such that $k/n<t$ we have $k+1\leq \lceil t \rceil n$. Thus, one gets $(k+1)y\leq \lceil t \rceil ny\leq k (\lceil t \rceil x)$. Consequently, we obtain $\varphi_2 (y)\leq \lceil t \rceil\varphi_2\varphi_1 (x)$, as desired.
\end{proof}

\begin{prp}
Let $\varphi_1\colon S_1\to S_2$ and $\varphi_2\colon S_2\to T$ be generalized $\Cu$-morphisms. Assume that $\varphi_1$ is almost divisible, and that $\varphi_2$ is almost unperforated. Then, for any $x\in S_1$, there exists a generalized $\Cu$-morphism $\alpha_x\colon Z\to T$ such that $\alpha_x (1)=\varphi_2\varphi_1 (x)$.
\end{prp}
\begin{proof}
For every $n\in Z_c$, set $\alpha_x (n):=n\varphi_2 \varphi_1 (x)$. For each $t\in Z_\soft$, define 
\[
	\alpha_x (t):= \sup \Phi (t,\varphi_1(x)),
\]
which exists by \cref{prp:SupZabso}. 

Note that, for any $s\geq t$ in $Z_\soft$, one has $\Phi (t,\varphi_1 (x))\subseteq \Phi (s,\varphi_1 (x))$. This implies $\alpha_x (t)\leq \alpha_x (s)$. Further, it follows from \cref{prp:SupZabso} that $\alpha_x (\sigma (1))\leq \varphi_2\varphi_1 (1)=\alpha_x (1)$. Additionally, for any $\varepsilon >0$, take any $x'\in S_1$ such that $x'\ll x$ and let $n\in\NN$ be such that $1<(n+2)/n<1+\varepsilon$. Using almost divisibility, there exists $y$ in $\mu (((n+2),n),\varphi_1 (x'),\varphi_1 (x))$. One has $(n+2)\varphi_1 (x')\leq (n+1)y$. Using almost unperforation, $\varphi_2 \varphi_1 (x')\leq \varphi_2 (y)\leq \alpha_x (1+\varepsilon )$ and, by taking suprema on $x'$, we deduce $\alpha_x (1)=\varphi_2 \varphi_1 (x)\leq \alpha_x (1+\varepsilon )$.

The arguments above show that $\alpha_x (\sigma (1))\leq \alpha_x (1)\leq \alpha_x (1+\varepsilon )$ for every $\varepsilon >0$. Further, note that $\alpha_x|_{Z_c}$ is trivially an order-preserving monoid morphism. We will now prove that $\alpha_x |_{Z_\soft}$ is a generalized \CuMor{}. \cref{prp:ExtensionZtoW} will then imply that $\alpha_x$ is a generalized \CuMor{}. We have already shown that $\alpha_x|_{Z_\soft}$ is order-preserving, so it suffices to prove that the map preserves suprema and addition.

To show that it preserves suprema, note that any increasing sequence $(t_d)_d$ in $Z_\soft$ satisfies $\cup_{d}\Phi (t_d,\varphi_1(x))=\Phi (\sup_d t_d,\varphi_1(x))$. This proves that $\alpha_x$ preserves suprema in $Z_\soft$.

To see that the map is additive, let $t_1,t_2\in Z_\soft=[0,\infty ]$. First, for each $i=1,2$, let $y_i\in \mu ((k_i,n_i),0,\varphi_1 (x))$ for some $k_i,n_i\in\NN$ such that
\[
	\frac{k_i}{n_i}<t_i.
\]

Take $y_i'\in S_2$ such that $y_i'\ll y_i$ for $i=1,2$, and let $x_0\in S_1$ be such that 
\[
	x_0\ll x,\andSep 
	y_i'\in\mu ((k_i,n_i),0,\varphi_1 (x_0)).
\]

Choose $k,n\in\NN$ such that 
\[ 
	\frac{k_1}{n_1}+\frac{k_2}{n_2} = \frac{k_1 n_2+k_2 n_1}{n_1 n_2}< \frac{k}{n+1}, \andSep
	\frac{k}{n}<t_1+t_2.
\]

Using that $\varphi_1$ is almost divisible, find $y\in \mu ((k,n),\varphi_1 (x_0),\varphi_1 (x))$. Note that $y_1'+y_2'\in \mu ((k_1 n_2+k_2 n_1,n_1 n_2),0,\varphi_1 (x_0))$. By \cref{lma:AlmUnpfMap} (1), one gets $\varphi_2 (y_1'+y_2')\leq \varphi_2 (y)$. Thus, we get 
\[
	\varphi_2 (y_1')+\varphi_2 (y_2')\leq \varphi_2 (y)\leq \alpha_x (t_1+t_2)
\]
and, since this holds for every choice of $y_1',y_2'$, one obtains $\varphi_2 (y_1)+\varphi_2 (y_2)\leq \alpha_x (t_1+t_2)$. Taking suprema now on $y_1,y_2$, we have $\alpha_x (t_1)+\alpha_x (t_2)\leq \alpha_x (t_1+t_2)$.

Conversely, take $y\in \mu ((k,n),0,\varphi_1 (x))$ for $k,n$ with $k/n<t_1+t_2$ and $x'\ll x$. Find $k_i,t_i\in\NN$ such that $\frac{k_i}{n_i}<t_i$ and
\[
	\frac{k}{n}< \frac{k_1}{n_1+1}+\frac{k_2}{n_2+1} =
	\frac{k_1(n_2+1) + k_2(n_1+1)}{(n_1+1)(n_2+1)}
	.
\]

Proceeding as before, take $y'\ll y$ and find $x'$ such that $x'\ll x$ and $y'\in \mu ((k,n),0,\varphi_1 (x'))$. Find $y_i\in \mu ((k_i,n_i),\varphi_1 (x'),\varphi_1 (x))$. In particular, we have
\[
	\begin{split}
	(k_1(n_2+1) + k_2(n_1+1))ny' &\leq (k_1(n_2+1) + k_2(n_1+1))k\varphi_1(x')\\
	&\leq (n_1+1)(n_2+1)k(y_1+y_2).
	\end{split}
\]

Since $\varphi_2$ is almost unperforated, we obtain $\varphi_2(y')\leq \varphi_2(y_1)+\varphi_2(y_2)\leq \alpha_x (t_1)+\alpha_x (t_2)$. Taking suprema on $y'$, and then on $y$, we deduce $\alpha_x (t_1+t_2)\leq \alpha_x (t_1)+\alpha_x (t_2)$, as desired. This proves that $\alpha_x$ is additive in $Z_\soft$.

\cref{prp:ExtensionZtoW} shows that $\alpha_x\colon Z\to T$ is a generalized \CuMor{}.
\end{proof}

\begin{thm}\label{prp:McDuffZ}
Let $\varphi_1\colon S_1\to S_2$ and $\varphi_2\colon S_2\to T$ be (generalized) $\Cu$-morphisms. Assume that $\varphi_1$ is almost divisible, and that $\varphi_2$ is almost unperforated. Then, there exists a (generalized) \CuMor{} $\beta\colon S_1\otimes Z\to T$ such that the following diagram commutes 
\[
 \xymatrix{
     S_1 \ar[r]^{\varphi_1} \ar[rd]_{- \otimes 1} & S_2\ar[r]^{\varphi_2} & T \\
     & S_1\otimes Z \ar[ru]_{\beta}
   } 
\]
\end{thm}
\begin{proof}
We define the map $\alpha\colon S_1\times Z\to T$ by
\[
	\alpha (x,t) := \alpha_x (t),
\]
which satisfies $\alpha (x,1)= \varphi_2\varphi_1 (x)$.

We will now prove that $\alpha$ is a generalized $\Cu$-bimorphism. Using \cite[Theorem~6.3.3~(1)]{AntPerThi18TensorProdCu}, this will imply the existence of $\beta\colon S\otimes Z\to T$ with the required properties.

Note that $\alpha (x,\cdot )=\alpha_x$ is a generalized \CuMor{} by the results above. Further, $\alpha (\cdot , n)$ is trivially a generalized \CuMor{} for every $n\in Z_c$. Set $\gamma_t (x):=\alpha (x , t)$ for every $t\in Z_\soft$, and let us show that $\gamma_t$ is a generalized \CuMor{}.

To see that $\gamma_t$ preserves order, take $x_0,x\in S_1$ such that $x_0\leq x$. Clearly, one has
\[
	\Phi (t,\varphi_1 (x_0))\subseteq \Phi (t,\varphi_1 (x))
\]
and thus $\gamma_t (x_0)=\alpha_{x_0} (t)\leq \alpha_x (t)=\gamma_t (x)$.

Now take $(x_d)_d$ in $S_1$ increasing with supremum $x$. Then, for any element $y\in \mu ((k,n),0,\varphi_1 (x))$ with $k/n<t$, take $y'\ll y$ and find $d\in\NN$ such that $y'\in \mu ((k,n),0,\varphi_1 (x_d))$. This shows that $\varphi_2 (y')\leq \gamma_t (x_d)$ and, consequently, $\varphi_2 (y)\leq \sup_d \gamma_t (x_d)$. Taking suprema, one gets $\gamma_t (x)\leq \sup_d \gamma_t (x_d)$. Since $\gamma_t$ is order-preserving, we also obtain $\sup_d \gamma_t (x_d)\leq \gamma_t (x)$. In other words, $\gamma_t$ preserves suprema of increasing sequences.

To prove that $\gamma_t$ is superadditive, take $x_1,x_2\in S_1$ and let $y_1,y_2\in S_2$ be such that $y_i\in\mu ((k_i,n_i),0,\varphi_1 (x_i))$ for some $k_i,n_i$'s such that $k_i/n_i<t$. Find $k,n\in\NN$ such that $k_i/n_i<k/(n+1)$ and $k/n<t$. Take $y_i'\in S_2$ such that $y_i'\ll y_i$, and let $x_i'\in S_1$ be such that $x_i'\ll x_i$ and $y_i'\in\mu ((k_i,n_i),0,\varphi_1 (x_i'))$. Using almost divisibility of $\varphi_1$, find $z_i\in\mu ((k,n),\varphi_1 (x_i'),\varphi_1(x_i))$. By \cref{lma:AlmUnpfMap} (1), we get $\varphi_2 (y_i')\leq \varphi_2 (z_i)$. Note that we have $z_1+z_2\in \mu ((k,n),0,\varphi_1 (x_1+x_2))$. Thus, one gets
\[
	\varphi_2 (y_1')+\varphi_2 (y_2')\leq \varphi_2 (z_1)+\varphi_2 (z_2)=\varphi_2 (z_1+z_2)\leq \alpha_{x_1+x_2}(t)=\gamma_t (x_1+x_2).
\]

Taking suprema on $y_1'$ and $y_2'$, this implies $\varphi_2 (y_1)+\varphi_2 (y_2)\leq \gamma_t (x_1+x_2)$. Taking now suprema on $y_1,y_2,k$ and $n$ gives $\gamma_t (x_1)+\gamma_t (x_2)\leq \gamma_t (x_1+x_2)$.

Conversely, to prove subadditivity, let $y\in \mu ((k,n),0,\varphi_1 (x_1+x_2))$. Take $y'\ll y$ and let $x_i'\ll x_i$ be such that $y'\in \mu ((k,n),0,\varphi_1 (x_1'+x_2'))$. Find $l,m\in\NN$ such that $k/n<l/(m+1)$ and $l/m<t$. Find $y_i\in \mu ((l,m),\varphi_1 (x_i'),\varphi_1 (x_i))$. Then, $y_1+y_2\in \mu ((l,m),\varphi_1 (x_1'+x_2'),\varphi_1 (x_1+x_2))$. By \cref{lma:AlmUnpfMap}, one obtains $\varphi_2 (y')\leq \varphi_2 (y_1)+\varphi_2 (y_2)$. Again, this implies $\varphi_2 (y)\leq \gamma_t (x_1)+\gamma_t (x_2)$ and, consequently, $\gamma_t (x_1+x_2)\leq \gamma_t (x_1)+\gamma_t (x_2)$.

We have shown that each coordinate of $\alpha$ is a generalized \CuMor{}. In particular, we know that there exists a generalized \CuMor{} $\beta \colon S_1\otimes Z\to T$ with the desired properties; see \cite[Lemma~6.3.2,~Theorem~6.3.3]{AntPerThi18TensorProdCu}.

Now assume that $\varphi_1$ and $\varphi_2$ are \CuMor{s}. To prove that $\alpha$ is in fact a $\Cu$-bimorphism, take $t',t\in Z$ and $x',x\in S$ such that $t'\ll t$ and $x'\ll x$. We have to show that $\alpha (x',t')\ll \alpha (x,t )$ . If $t'$ or $t$ are in $\NN$, we may assume $t=t'$. In this case, one has $\alpha (x',t)=t\varphi_2\varphi_1 (x')\ll t\varphi_2\varphi_1 (x)=\alpha (x,t)$ because both $\varphi_1$ and $\varphi_2$ are \CuMor{s}. Finally, assume $t',t\in (0,\infty ]$. Take $x_1,x_2\in S_1$ such that $x'\ll x_1\ll x_2\ll x$.

Find $l_1,l_2,m_1,m_2\in\NN$ such that
\[
	t'<\frac{l_1}{m_1+1},\quad 
	\frac{l_1}{m_1}<\frac{l_2}{m_2+1},\quad 
	\frac{l_2}{m_2}<t.
\]

By almost divisibility of $\varphi_1$, there exist elements $y_1,y_2\in S_2$ such that $y_1\in \mu ((l_1,m_1),\varphi_1 (x'),\varphi_1 (x_1))$ and $y_2\in\mu ((l_2,m_2),\varphi_1 (x_2),\varphi_1 (x))$. By \cref{lma:AlmUnpfMap} (2), one gets $\varphi_2 (y_1)\ll \varphi_2 (y_2)\leq \alpha (x,t)$.

Now note that for every $y\in \mu ((k,n),0,\varphi_1 (x'))$ with $k/n<t'$, one has $k/n<l_1/(m_1+1)$. Thus, another application of \cref{lma:AlmUnpfMap} (1) gives $\varphi_2 (y)\leq \varphi_2 (y_1)$. In other words, $\alpha (x',t')\leq \varphi_2 (y_1)$. Since we already know that $\varphi_2 (y_1)\ll \alpha (x,t)$, one gets $\alpha (x',t')\ll \alpha (x,t)$, as desired.

Now \cite[Theorem~6.3.3]{AntPerThi18TensorProdCu} shows that $\beta\colon S_1\otimes Z\to T$ is a \CuMor{} with the desired properties.
\end{proof}

\begin{cor}
Let $\varphi\colon S\to T$ be a \CuMor{}. Then,
\begin{itemize}
\item[(i)] if $S$ is almost divisible and $\varphi$ is almost unperforated, $\varphi$ factorizes through $S\otimes Z$. 
\item[(ii)]  if $T$ is almost unperforated and $\varphi$ is almost divisible, $\varphi$ factorizes through $S\otimes Z$. 
\end{itemize}
\end{cor}
\begin{proof}
For (i), consider the composition of maps $S\to S\to T$ and apply \cref{prp:McDuffZ}. For (ii), consider $S\to T\to T$ and apply \cref{prp:McDuffZ}.
\end{proof}

\cref{prp:McDuffZ} above provides a partial answer to \cref{qst:Pure2}:

\begin{thm}\label{thm:MainCuA}
 Let $\theta_1\colon A_1\to A_2$ and $\theta_2\colon A_2\to B$ be pure \stHom{s}. Then, there exists a \CuMor{} $\beta$ such that the following diagram commutes
 \[
 \xymatrix{
     \Cu (A_1) \ar[rr]^{\Cu (\theta_2\theta_1)} \ar[rd]_{-\otimes 1} && \Cu (B) \\
     & \Cu (A_1)\otimes \Cu (\mathcal{Z}) \ar[ru]_{\beta}
   } 
\]
\end{thm}

One also obtains the following proposition, which answers \cref{qst:Pure2} completely when the initial domain is an AF-algebra and the codomain is of stable rank one. We believe that this result may be far more general, but new or tinkered techniques need to be developed in order to do so; see \cref{qst:Ext}.

\begin{prp}\label{prp:AFPureFactor}
Let $A_1$ be a unital AF-algebra, and let $B$ be a unital \ca{} of stable rank one. Let $\theta_1\colon A_1\to A_2$ and $\theta_2\colon A_2\to B$ be unital, pure *-homo\-mor\-phisms. Then, $\theta_2\theta_1$ factors, up to approximately unitarily equivalence, through $A\otimes\mathcal{Z}$.
\end{prp}
\begin{proof}
The induced composition of \CuMor{s} factorizes through $\Cu (A)\otimes \Cu (\mathcal{Z})$ by \cref{prp:McDuffZ}. It follows from \cite[Proposition~6.4.13]{AntPerThi18TensorProdCu} that $\Cu (A)\otimes \Cu (\mathcal{Z})\cong \Cu (A\otimes \mathcal{Z})$.

Further, the \ca{} $A\otimes \mathcal{Z}$ is an inductive limit of 1-dimensional NCCW complexes with trivial $K_1$-group. The result now follows from \cite[Theorem~1.0.1]{Rob12Class}.
\end{proof}

Note that there are two obstructions to generalizing \cref{prp:AFPureFactor}. First, the Cuntz semigroup tensor product does not generally behave well with its $C^*$-algebraic counterpart. For example, it is known that $\Cu (C[0,1])\otimes \Cu (\mathcal{Z})\not\cong \Cu (C[0,1]\otimes \mathcal{Z})$; see \cite[Proposition~6.4.4]{AntPerThi18TensorProdCu}. Secondly, the only currently available result for lifting \stHom{s} is Robert's \cite[Theorem~1.0.1]{Rob12Class}. However, it is conceivable that such result can be generalized whenever $A,B$ are sufficiently noncommutative (say, if $A,B$ are simple) and $\Cu (\theta )$ maps every Cuntz class to a strongly soft class. The following question is related to the first of the two obstructions.

\begin{qst}\label{qst:Ext}
 Let $A,B$ be \ca{s}, and let $\psi\colon\Cu (A)\otimes\Cu (\mathcal{Z})\to \Cu (B)$ be a \CuMor{}. When does there exist a \CuMor{} $\rho\colon\Cu (A\otimes \mathcal{Z})\to \Cu (B)$ such that $\psi ([a]\otimes 1)=\rho ([a\otimes 1])$?
\end{qst}

As shown by Winter in \cite[Corollary~7.4]{Win12NuclDimZstable}, a separable, unital, simple, non-elementary \ca{} of locally finite nuclear dimension is pure if and only if it is $\mathcal{Z}$-stable. In analogy to this result, one may ask:

\begin{qst}\label{qst:TWMaps}
Let $A,B$ be \ca{s}. Under which conditions on $A$ and $B$ does every pure \stHom{} $\theta\colon A\to B$ factor (in a suitable sense) through a $\mathcal{Z}$-stable \ca{}?
\end{qst}

Restricting the codomain in \cref{prp:AFPureFactor} further, we obtain a first answer to \cref{qst:TWMaps}.

\begin{cor}\label{cor:AFStrict}
Let $A$ be a unital AF-algebra, and let $B$ be a unital \ca{} of stable rank one and with strict comparison. Let $\theta\colon A\to B$ be a unital, pure \stHom{}. Then, $\theta$ factors up to approximately unitarily equivalence through $A\otimes\mathcal{Z}$.
\end{cor}

\begin{rmk}
Following the ideas from \cref{dfn:ZMultPure}, one could also define other $\Cu$-like notions for morphisms, such as algebraicity and (weak) $(2,\omega)$-divisibility.
\end{rmk}

\section{\texorpdfstring{Soft and rational \stHom{s}}{Soft and rational *-homomorphisms}}\label{subsec:WMult}

In this last section we exploit \cref{thm:MainCuA} in two cases of interest: Pure maps with a soft image (\cref{dfn:WMult}), and rational maps (\cref{dfn:RatCuMor}). These notions are meant to generalize tensorial absorption, at a Cuntz semigroup level, of the Jacelon-Razak algebra and UHF-algebras respectively. In contrast to \cref{thm:MainCuA}, $\Cu$-tensor products and $C^*$-tensor products of such algebras do behave nicely. This allows us to show that a composition of maps always factors (at the level of $\Cu$) through $A\otimes M_q$ and $A\otimes \mathcal{W}$ respectively; see Theorems \ref{thm:MainCuAMq} and \ref{thm:MainCuAW}.

\subsection{\texorpdfstring{$q$}{q}-rational morphisms}

Given a supernatural number $q$ such that $q=q^2$ and $q\neq 1$, let $M_q$ denote the UHF-algebra associated to $q$. As shown in \cite[Section~7.4]{AntPerThi18TensorProdCu}, $\Cu (M_q)\cong K_q\sqcup (0,\infty ]$ where $K_q$ is the subset of $\QQ_+$ formed by the elements of the form $\frac{k}{n}$ with $k,n$ coprime and $n$ a divisor of $q$.

Adapting \cite[Definition~7.4.6]{AntPerThi18TensorProdCu} to our setting, we define:

\begin{dfn}\label{dfn:RatCuMor}
Let $\varphi\colon S\to T$ be a generalized \CuMor{}, and let $q$ be a supernatural number as above. We will say that $\varphi$ is \emph{$q$-rational} if it is both
\begin{itemize}
\item[(i)] \emph{$q$-divisible}, that is, if for every $x\in S$ and every finite divisor $n$ of $q$ there exists $y\in T$ such that $\varphi (x)=ny$.
\item[(ii)] \emph{$q$-unperforated}, that is, if whenever $nx\leq ny$, for some finite divisor $n$ of $q$, one has $\varphi (x)\leq \varphi (y)$. 
\end{itemize}
\end{dfn}

As shown in \cite[Theorem~7.4.10]{AntPerThi18TensorProdCu}, a \CuSgp{} $S$ tensorially absorbs $\Cu (M_q)$ if and only if $S$ is $q$-divisible and $q$-unperforated. Thus, examples of morphisms satisfying the two conditions above include all morphisms whose domain or codomain absorbs $\Cu (M_q)$ tensorially.

\begin{prp}\label{prp:McDuffMq}
Let  $q$ be a supernatural number such that $q=q^2$ and $q\neq 1$. Let $\varphi_1\colon S_1\to S_2$ be a $q$-divisible (generalized) $\Cu$-morphism and let $\varphi_2\colon S_2\to T$ be a $q$-unperforated (generalized) $\Cu$-morphism. Then, there exists a (generalized) \CuMor{} $\gamma\colon S\otimes \Cu (M_q)\to T$ such that $\gamma (x\otimes 1)=\varphi_2\varphi_1(x)$.
\end{prp}
\begin{proof}
We mimic the approach of \cite[Theorem~7.4.10]{AntPerThi18TensorProdCu}.

First, note that for every $x\in S$ and every $n$ divisor of $q$ there exists a unique element $y\in\varphi_2 (S_2)$ such that $y=\varphi_2 (z)$ with $\varphi_1 (x)=nz$. Indeed, existence follows from $q$-divisibility of $\varphi_1$, while uniqueness is given by $\varphi_2$. Thus, the map $\omega_n\colon S\to T$ given by $x\mapsto y$ is well-defined. It is readily checked that $\omega_n$ is a (generalized) \CuMor{} whenever $\varphi_1$ and $\varphi_2$ are.

Further, it follows from \cref{prp:McDuffZ} that there exists a (generalized) $\Cu$-bi\-mor\-phism $\alpha\colon S\times Z\to T$ such that $\alpha (x,1)=\varphi_2\varphi_1(x)$. Now, define a  (generalized) $\Cu$-bimorphism $\alpha_q\colon S\times (K_q\sqcup (0,\infty ])\to T$ as follows: Given $t\in (0,\infty]$, simply set $\alpha_q (x,t):=\alpha (x,t)$. Else, if $t=\frac{k}{n}$ for some (unique) coprime pair $k,n$ with $n$ a divisor of $q$, set $\alpha_q (x,t):=k\omega_n (x)$. By construction, $\alpha_q (\cdot ,t)$ is a generalized \CuMor{}.

A proof analoguous to that of \cref{prp:ExtensionZtoW} shows that $\alpha_q (x,\cdot )$ is a genera\-li\-zed \CuMor{} if and only if $\alpha_q (x,\cdot )|_{(0,\infty ]}$ is a generalized \CuMor{} and $\alpha_q (x,\sigma (\frac{1}{n}) )\leq \alpha_q (x,\frac{1}{n} )\leq \alpha_q (x,\sigma (\frac{1}{n}) +\varepsilon)$ for every $\varepsilon >0$ and any $n$ dividing $q$. Note that the first condition is satisfied by construction of $\alpha$, while the second condition is readily checked after a careful examination of $\alpha_q$. Thus, \cite[Theorem~6.3.3~(1)]{AntPerThi18TensorProdCu} implies the existence of the map $\gamma$ with the required properties.
\end{proof}

We know from \cite[Proposition~7.6.3]{AntPerThi18TensorProdCu} that $\Cu (A\otimes M_q)\cong \Cu (A)\otimes \Cu (M_q)$ always. Thus, one gets the following result, where \stHom{} and \CuMor{} can be changed to cpc. order-zero map and generalized \CuMor{} respectively.

\begin{thm}\label{thm:MainCuAMq}
 Let $\theta_1\colon A_1\to A_2$ and $\theta_2\colon A_2\to B$ be \stHom{s}. Assume that $\theta_1$ is $q$-divisible and that $\theta_2$ is $q$-unperforated. Then, there exists a \CuMor{} $\beta\colon\Cu (A\otimes M_q)\to\Cu (B)$ such that $\Cu (\theta_2\theta_1)[a]=\beta([a\otimes 1])$ for each $[a]\in\Cu (A)$.
\end{thm}

Applying Robert's classification result (\cite[Theorem~1.0.1]{Rob12Class}), one obtains:

\begin{cor}\label{cor:RobMq}
Retain the above assumptions. Assume that $A_1$ is a unital \ca{} stably isomorphic to an  inductive limit of $1$-dimensional NCCW-complexes with trivial $K_1$-group, that $B$ is unital and of stable rank one, and that $\theta_1$ and $\theta_2$ are unital. Then, $\theta_2\theta_1$ factors, up to approximate unitary equivalence, through $A\otimes M_q$.
\end{cor}

\subsection{Soft maps}

Recall that the Cuntz semigroup of the Jacelon-Razak algebra $\mathcal{W}$ is isomorphic to the monoid $[0,\infty ]$. As shown in \cite[Theorem~7.5.4]{AntPerThi18TensorProdCu}, every element in a pure \CuSgp{} $S$ is strongly soft (see \cref{pgr:Soft} below) if and only if $S\cong S\otimes \Cu (\mathcal{W})$. We show the analogue of \cref{prp:McDuffZ} in \cref{prp:McDuffW}.

\begin{pgr}[Soft elements and strongly soft Cuntz classes]\label{pgr:Soft}
Let $S$ be a \CuSgp{}, and let $x\in S$. Recall from \cite[Definition~4.2]{ThiVil23Soft} that $x$ is \emph{strongly soft} if, for any $x'\in S$ such that $x'\ll x$, there exists $t\in S$ such that $x'+t\leq x\leq \infty t$. Making an abuse of notation, the set of strongly soft elements is usually denoted by $S_{\soft}$.

Let $A$ be a stable \ca{}. Under the presence of sufficient non-commutativity on $A$ (for example, if $A$ has the Global Glimm Property), a positive element $a\in A_+$ has a strongly soft Cuntz class if and only if $a$ is Cuntz equivalent to a \emph{soft element}, that is, an element $b\in A_+$ such that no nontrivial quotient of $\overline{bAb}$ is unital; this equivalence is proved in \cite{AsVaThiVil23Soft}.
\end{pgr}

\begin{dfn}\label{dfn:WMult}
 Let $\varphi\colon S\to T$ be a generalized \CuMor{}. We will say that $\varphi$ is \emph{soft} if $\varphi (S)\subseteq T_\soft$.

Further, $\varphi$ will be said to have \emph{$\Cu (\mathcal{W})$-multiplication} if it is soft and has $\Cu (\mathcal{Z})$-multiplication.
\end{dfn}

\begin{rmk}$ $
\begin{enumerate}
\item Let $\theta\colon A\to B$ be a \stHom{} between stable \ca{s}, and assume that $B$ satisfies the Global Glimm Property. Then, $\Cu(\theta )$ has $\Cu (\mathcal{W})$-multiplication if and only if $\theta$ is pure and $\theta (a)$ is Cuntz equivalent to a soft element for every $a\in A_+$.
\item Note the analogue statements of \cref{prp:LSkMult} and Propositions \ref{prp:idMult}, \ref{prp:CompMult} and \ref{prp:OinftyMultGabe} also work with $\Cu (\mathcal{W})$-multiplication instead of $\Cu (\mathcal{Z})$-multi\-pli\-ca\-tion.
\end{enumerate}
\end{rmk}

\begin{exa}
By (the analogue of) \cref{prp:CompMult}, any \stHom{} $A\to B$ that factorizes through $A\otimes \mathcal{W}$ induces a \CuMor{}  with $\Cu (\mathcal{W})$-multiplication.

As noted in \cite[Remark~3.21]{Gabe2020ANewProofKirchberg}, the infinite repeat $\phi\otimes 1_{\mathcal{M}(\mathcal{K})}\colon A\to\mathcal{M}(B\otimes\mathcal{K})$ of any \stHom{} of the form $A\to \mathcal{M}(B)$ factorizes through $\mathcal{M}(B)\otimes \mathcal{O}_2$.

Since every $\mathcal{O}_2$-stable \ca{} is purely infinite by \cite[Theorem~5.11]{KirRor00PureInf}, its Cuntz semigroup has $\{ 0,\infty \}$-multiplication and, also, $[0,\infty ]$-multiplication. Thus, it follows that $\Cu (\phi\otimes 1_{\mathcal{M}(\mathcal{K})})$ always has $\Cu (\mathcal{W})$-multiplication regardless of our choice of $A$ and $B$.
\end{exa}


\begin{thm}\label{prp:McDuffW}
Let $\varphi_1\colon S_1\to S_2$ and $\varphi_2\colon S_2\to T$ be (generalized) \CuMor{s}. Assume that $\varphi_1$ is soft and almost divisible, and that $\varphi_2$ is almost unperforated. Then, there exists a (generalized) \CuMor{} $\gamma\colon S_1\otimes [0,\infty ]\to T$ such that ${\gamma (x\otimes 1)=\varphi_2\varphi_1 (x)}$.
\end{thm}
\begin{proof}
Let $\beta\colon S_1\otimes Z\to T$ be the map constructed in \cref{prp:McDuffZ}, and denote by $\gamma$ the restriction of $\beta$ to $(S\otimes Z)_\soft\cong S\otimes [0,\infty ]$.

Let $x\in S_1$ and take $x'\in S_1$ such that $x'\ll x$. Then, since $\varphi_1 (x)$ is soft, there exists $n\in\NN$ such that $(n+1)\varphi_1 (x')\leq n\varphi_1 (x)$; see, for example, \cite[Proposition~4.5]{ThiVil23Soft}. This implies that $\varphi_1 (x')\in\mu ((n,n+1),0,\varphi_1 (x))$ and, in particular, that $\varphi_2\varphi_1 (x')\in \Phi (1,\varphi_1 (x))$. Thus, one has $\varphi_2 \varphi_1 (x')\leq \gamma (x\otimes 1)$ for every $x'$. Consequently, $\varphi_2\varphi_1 (x)\leq \gamma (x\otimes 1)$. Further, note that one gets
\[
\gamma (x\otimes 1)=\beta (x\otimes 1_{Z_\soft })\leq \beta (x\otimes 1_Z)=\varphi_2\varphi_1 (x)
,
\]
that is, $\gamma (x\otimes 1)=\varphi_2 \varphi_1 (x)$, as desired.
\end{proof}

Combining \cite[Theorem~5.1.2]{Rob13b} and \cite[Proposition~7.6.3]{AntPerThi18TensorProdCu}, one has that $\Cu (A\otimes \mathcal{W})\cong \Cu (A)\otimes [0,\infty]$; hence, we obtain the following result.

\begin{thm}\label{thm:MainCuAW}
Let $\theta_1\colon A_1\to A_2$ be a soft and pure \stHom{}, and let $\theta_2\colon A_2\to B$ be a pure \stHom{}. Then, there exists a \CuMor{} $\beta\colon\Cu (A\otimes \mathcal{W})\to\Cu (B)$ such that $\Cu (\theta_2\theta_1)[a]=\beta([a\otimes 1])$ for each $[a]\in\Cu (A)$.
\end{thm}
\providecommand{\etalchar}[1]{$^{#1}$}
\providecommand{\bysame}{\leavevmode\hbox to3em{\hrulefill}\thinspace}
\providecommand{\noopsort}[1]{}
\providecommand{\mr}[1]{\href{http://www.ams.org/mathscinet-getitem?mr=#1}{MR~#1}}
\providecommand{\zbl}[1]{\href{http://www.zentralblatt-math.org/zmath/en/search/?q=an:#1}{Zbl~#1}}
\providecommand{\jfm}[1]{\href{http://www.emis.de/cgi-bin/JFM-item?#1}{JFM~#1}}
\providecommand{\arxiv}[1]{\href{http://www.arxiv.org/abs/#1}{arXiv~#1}}
\providecommand{\doi}[1]{\url{http://dx.doi.org/#1}}
\providecommand{\MR}{\relax\ifhmode\unskip\space\fi MR }
\providecommand{\MRhref}[2]{%
  \href{http://www.ams.org/mathscinet-getitem?mr=#1}{#2}
}
\providecommand{\href}[2]{#2}

\end{document}